\documentclass{conm-p-l}

\usepackage{amsmath}
\usepackage{amssymb}
\usepackage{amscd}
\input xy
\xyoption{all}

\usepackage{color}
\usepackage{hyperref} % hyperlink ref/cite

\newtheorem{theorem}{Theorem}[section]
\newtheorem{lemma}[theorem]{Lemma}
\newtheorem{prop}[theorem]{Proposition}
\newtheorem{deflemma}[theorem]{Definition and Lemma}
\newtheorem{cor}[theorem]{Corollary}

\theoremstyle{definition}
\newtheorem{definition}[theorem]{Definition}
\newtheorem{example}[theorem]{Example}

\theoremstyle{remark}

\numberwithin{equation}{section}

\newcommand{\Cal}[1]{{\mathcal #1}}
\newcommand{\End}{\mbox{\rm End}}

\newcommand{\Hom}{\mbox{\rm Hom}}

\newcommand{\ann}{\mbox{\rm ann}}

\newcommand{\Ann}{\mbox{\rm Ann}}

\newcommand{\coker}{\mbox{\rm coker}}

\newcommand{\Z}{\mathbb{Z}}

\newcommand{\f}{\frac}
\newcommand{\bF}{ \mathbb{F} }
\newcommand{\bP}{ \mathbb{P} }
\newcommand{\bQ}{ \mathbb{Q} }
\newcommand{\bZ}{ \mathbb{Z} }

\newcommand{\cE}{ \mathcal{E} }
\newcommand{\cF}{ \mathcal{F} }

\newcommand{\sL}{ \mathsf{L} }

\newcommand{\fp}{ \mathfrak{p} }
\newcommand{\fP}{ \mathfrak{P} }
\newcommand{\lbar}{ \overline }

\newcommand{\enumequivi}{%
  \renewcommand{\theenumi}{(\alph{enumi})}%
  \renewcommand{\labelenumi}{\theenumi}%
}

\begin{document}

 \title[Cyclically presented modules, projective covers and factorizations]{Cyclically presented modules, projective covers \protect \\ and factorizations}
    \author[Alberto Facchini]{Alberto Facchini}
    \address{Dipartimento di Matematica, Universit\`a di Padova, 35121 Padova, Italy}
   \email{facchini@math.unipd.it}
\thanks{Partially
supported by Universit\`a di Padova (Progetto di ricerca di Ateneo CPDA105885/10 ``Differential graded categories''
 and Progetto ex 60\% ``Anelli e categorie di moduli'') and Fondazione Cassa di Risparmio di Padova e Rovigo (Progetto di Eccellenza ``Algebraic structures and their applications''.)}
 
\author[Daniel Smertnig]{Daniel Smertnig}
\address{Institut f\"ur Mathematik und Wissenschaftliches Rechnen \\
Karl-Franzens-Uni\-ver\-si\-t\"at Graz \\
Heinrichstra\ss e 36\\
8010 Graz, Austria}
\email{daniel.smertnig@uni-graz.at}
\thanks{The second author is supported by the Austrian Science Fund (FWF), W1230.}

\author[Nguyen Khanh Tung]{Nguyen Khanh Tung}
    \address{Dipartimento di Matematica, Universit\`a di Padova, 35121 Padova, Italy}
   \email{nktung@math.unipd.it}


\date{}

\begin{abstract} We investigate projective covers of cyclically presented modules, characterizing the rings over which every cyclically presented module has a projective cover as the rings $R$ that are Von Neumann regular modulo their Jacobson radical $J(R)$ and in which idempotents can be lifted modulo $J(R)$. Cyclically presented modules naturally appear in the study of factorizations of elements in non-necessarily commutative integral domains. One of the possible applications is to the modules $M_R$ whose endomorphism ring $E:=\End(M_R)$ is Von Neumann regular modulo $J(E)$ and in which idempotents lift modulo~$J(E)$. \end{abstract}
    \maketitle

\section{Introduction}

An $R$-module $M_R$ is said to be \emph{cyclically presented} if $M_R \cong R/aR$ for some $a \in R$. In this paper, we study some natural connections between cyclically presented $R$-modules, their submodules, their projective covers and factorizations of elements in the ring $R$. That is, we find some results on projective covers of cyclically presented modules and apply them to the study of factorizations of elements in a ring. In this way, we are naturally led to the class of $2$-firs. Recall that a ring $R$ is a $2$-fir if every right ideal of $R$ generated by at most $2$ elements is free of unique rank. This condition is right/left symmetric, and a ring $R$ is a $2$-fir if and only
 if it is a domain and the sum of any two principal right ideals with non-zero intersection is again a principal right ideal \cite[Theorem 1.5.1]{Cohn '85}. P.~M.~Cohn investigated factorization of elements in $2$-firs, applying the Artin-Schreier Theorem and the Jordan-H\"older-Theorem to the corresponding cyclically presented modules \cite{Cohn '85}. One of the main ideas developed in this paper is  to characterize the submodules of a cyclically presented module $M_R$ that, under a suitable cyclic presentation $\pi_M\colon R_R \to M_R$, lift to principal right ideals of $R$ that are generated by a left cancellative element (Lemmas~\ref{square}, \ref{piexact} and~\ref{cuo}). The key role is played by a class of cyclically presented submodules of a cyclically presented module $M_R$, which we call {\em $\pi_M$-exact submodules} of $M_R$. We show (Theorem~\ref{3.8}) that,  for every cyclically presented right $R$-module $M_R$ and every cyclic presentation $\pi_M\colon R_R \to M_R$ with non-zero kernel, the set of all cyclically presented $\pi_M$-exact submodules is closed under finite sums if and only if   $R$ is a $2$-fir. As we have said above, when sums and intersections of exact submodules are again exact submodules, we can use the Artin-Schreier and the Jordan-H\"older Theorems to study factorizations of elements. 

We also study the rings over which every cyclically presented module has a projective cover. We characterize these rings as the rings $R$ that are Von Neumann regular modulo their Jacobson radical $J(R)$ and in which idempotents can be lifted modulo $J(R)$ (Theorem~\ref{VNR}). Finally, in the last Section, we consider the modules $M_R$ whose endomorphism rings $E$ are Von Neumann regular modulo the Jacobson radical $J(E)$ and in which idempotents can be lifted modulo $J(E)$. In particular, this applies to the case in which the module $M_R$ in question is quasi-projective (Lemma~\ref{lemma:endo} and Proposition~\ref{vhlpp}). 

 Throughout the paper, $R$ will be an associative ring with identity $1_R\ne 0_R$ and we will denote by $U(R)$ its group of invertible elements. By an $R$-module, we always mean a unitary right $R$-module. 

\section{Generalities}

Let $R$ be a ring.  An element $a \in R$ is \emph{left cancellative} if, for all $b,c \in R$, $ab = ac$ implies $b=c$. Equivalently, $a \in R$ is left cancellative if it is non-zero and is not a left zero-divisor. A (non-necessarily commutative) ring $R$ is a \emph{domain} if every non-zero element is left cancellative (equivalently, if every non-zero element is right cancellative).
If $a\in R$, the right $R$-module homomorphism $\lambda_a\colon R_R \to aR, x \mapsto ax$, is an isomorphism if and only if $a$ is left cancellative. More precisely, $aR \cong R_R$ if and only if there exists a left cancellative element $a' \in R$ with $a'R=aR$. If $a, a' \in R$ are two left cancellative elements, then $aR = a'R$ if and only if $a = a'\varepsilon$ for some $\varepsilon \in U(R)$.

Let $a ,x_1, \ldots, x_n \in R \setminus U(R)$ be $n+1$ left cancellative elements and assume that $a = x_1 \cdot\ldots\cdot x_n$. If $\varepsilon_1,\ldots,\varepsilon_{n-1} \in U(R)$, then obviously also $a = (x_1 \varepsilon_1) \cdot (\varepsilon_1^{-1} x_2 \varepsilon_2) \cdot\ldots\cdot (\varepsilon_{n-1}^{-1}x_n)$. This gives an equivalence relation on finite ordered sequences of left cancellative elements whose product is $a$. More precisely, if $F_a:=\{\,(x_1, \ldots, x_n)\mid n\ge 1,\ x_i\in R \setminus U(R)$ is left cancellative for every $i=1,2,\dots,n$ and $a = x_1 \cdot\ldots\cdot x_n\,\}$, then the equivalence relation $\sim$ on $F_a$ is defined by $(x_1, \ldots, x_n)\sim (x'_1, \ldots, x_m')$ if $n=m$ and there exist $\varepsilon_1,\ldots,\varepsilon_{n-1} \in U(R)$ such that $x'_1=x_1 \varepsilon_1$, $x'_i=\varepsilon_{i-1}^{-1} x_i \varepsilon_i$ for all $i=2,\dots,n-1$ and $x'_n=\varepsilon_{n-1}^{-1}x_n$. In this paper, we call an equivalence class of $F_a$ modulo $\sim$ a \emph{factorization of $a$ up to insertion of units}. Notice that the factors need not be irreducible. When this causes no confusion, we will simply call a representative of such an equivalence class a factorization.

  A factorization $a = x_1 \cdot\ldots\cdot x_n$ gives rise to an ascending chain of principal right ideals, generated by left cancellative elements and containing $aR$:
\[
  aR \;\subsetneq\; x_1\cdot\ldots\cdot x_{n-1}R \;\subsetneq\; \ldots \;\subsetneq\; x_1R \subsetneq R,
\]
hence to an ascending chain of cyclically presented submodules \[
 0= aR/aR \;\subsetneq\; x_1\cdot\ldots\cdot x_{n-1}R/aR \;\subsetneq\; \ldots \;\subsetneq\; x_1R/aR \subsetneq R/aR
\]
of the cyclically presented $R$-module $R/aR$. Notice that $x_1\cdot\ldots\cdot x_{i-1}R/aR\cong R/x_i\cdot\ldots\cdot x_nR$ is cyclically presented because the elements $x_i$ are left cancellative.

The next lemma shows that, conversely, every chain of principal right ideals generated by left cancellative elements in $aR \subset R$, determines a factorization of $a$ into left cancellative elements, which is unique up to insertion of units.

\begin{lemma} \label{lemma:fact-chain}
  Let $a \in R$ be a left cancellative element, $aR = y_nR \subsetneq y_{n-1 }R \subsetneq \ldots \subsetneq y_{1} R \subsetneq y_0R = R$ be an ascending chain of principal right ideals of $R$, where $y_1,\ldots,y_{n-1} \in R$ are left cancellative elements, $y_0=1$ and $y_n=a$.
  For every $i =1,\ldots,n,$ let $x_i \in R$ be such that $y_{i-1}x_i = y_i$.
  Then $x_1, \ldots, x_{n}$ are left cancellative elements and $a= x_1 \cdot\ldots\cdot x_n$.

  Moreover, if $y_1', \ldots, y_{n-1}' \in R$ are also left cancellative elements with $y_i'R = y_iR$, $y_0'=1$ and $y_n'=a$, and we similarly define $x_i'$ by $y'_{i-1}x_i' = y_i' $ for every $i=1,2,\dots,n$, then there exist $\varepsilon_1,\ldots,\varepsilon_{n-1} \in U(R)$ such that $x'_1=x_1 \varepsilon_1$, $x'_i=\varepsilon_{i-1}^{-1} x_i \varepsilon_i$ for all $i=2,\dots,n-1$ and $x'_n=\varepsilon_{n-1}^{-1}x_n$.
\end{lemma}

\begin{proof}
Assume that $x_i$ is not left cancellative for some $i=1,2,\dots,n$. Then there exists $b \ne 0$ such that $x_i b=0$. Therefore $y_{i}b=y_{i-1} x_i b=0$. This is a contradiction because $y_{i}$ is left cancellative. Notice that $a=y_{n-1}x_n=y_{n-2}x_{n-1}x_n=\ldots =y_0x_1\ldots x_{n}=x_1\ldots x_n$.

Now if $y'_iR=y_iR$ for every $i=1,\dots,n-1$, then there exists $\varepsilon_1,\ldots,\varepsilon_{n-1} \in U(R)$ such that $y'_i=y_i \varepsilon_{i}$. Therefore $y'_{i-1}x'_i=y_{i-1}x_i\varepsilon_i=y'_{i-1}\varepsilon_{i-1}^{-1}x_i\varepsilon_i$. But $y'_{i-1}$ is left cancellative, so that $x'_i=\varepsilon_{i-1}^{-1}x_i\varepsilon_i$ for every $i=2,\dots,n-1$.

Moreover, $y_1=y_0x_1=x_1$ and, similarly, $y_1'= x_1'$, so that $y'_1=y_1 \varepsilon_1$ implies $x'_1=x_1 \varepsilon_1$.
Finally, $y_{n-1}x_n = y_n=a=y_n'=y_{n-1}'x_n'=y_{n-1} \varepsilon_{n-1}x_n'$. Thus $x_n= \varepsilon_{n-1}x_n'$ and $x_n'= \varepsilon_{n-1}^{-1}x_n$.
\end{proof}

As we have already said in the introduction, we will characterize, in Lemmas~\ref{piexact} and~\ref{cuo}, the submodules of cyclically presented modules $M_R$ that, under a suitable cyclic presentation $\pi\colon R_R \to M_R$, that is, a suitable epimorphism $\pi\colon R_R \to M_R$, lift to principal right ideals of $R$ generated by a left cancellative element.
The following lemma will prove to be helpful to this end.

\begin{lemma}\label{square}
  Let $A_R, B_R, M_R, N_R$ be modules over a ring $R$, $\pi_M\colon A_R \to M_R$ and $\pi_N\colon B_R \to N_R$ be epimorphisms, $\lambda\colon B_R\to A_R$ be a homomorphism and $\varepsilon\colon N_R \to M_R$ be a monomorphism such that $\pi_M \lambda = \varepsilon \pi_N$, so that there is a commutative diagram
$$\begin{array}{ccc} B_R & {\stackrel{\lambda}{\longrightarrow}} & A_R\\  {\scriptstyle \pi_N}\downarrow\phantom  {\scriptstyle \pi_N} & & \phantom{\scriptstyle \pi_M}\downarrow  {\scriptstyle \pi_M} \\  N_R &{\stackrel{\varepsilon}{\hookrightarrow}} 
& \phantom{.}M_R.\end{array}$$
 Then the following three conditions are equivalent:
    \begin{enumerate}
      \item[{\rm (a)}] \label{square:a} $\pi_M^{-1}(\varepsilon(N_R)) = \lambda(B_R)$.
      \item[{\rm (b)}] \label{square:b} $\lambda(\ker(\pi_N)) = \ker(\pi_M)$.
      \item[{\rm (c)}] \label{square:c} $\pi_M$ induces an isomorphism $\coker(\lambda) \to \coker(\varepsilon)$.
    \end{enumerate}

 If, moreover, $A_R', B_R'$ are right $R$-modules such that there exist isomorphisms $\varphi_A\colon A_R' \to A_R$ and $\varphi_B\colon B_R' \to B_R$, and one defines $\pi_N' := \pi_N \varphi_B$, $\pi_M' := \pi_M \varphi_A$ and $\lambda' := \varphi_A^{-1} \lambda \varphi_B$, then the three conditions {\rm (a)}, {\rm (b)} and {\rm (c)}  are equivalent also to the  the three conditions 
 \begin{enumerate}
      \item[{\rm (d)}] \label{square:d} $(\pi'_M)^{-1}(\varepsilon(N_R)) = \lambda'(B'_R)$.
      \item[{\rm (e)}] \label{square:e} $\lambda'(\ker(\pi'_N)) = \ker(\pi'_M)$.
      \item[{\rm (f)}] \label{square:f} $\pi'_M$ induces an isomorphism $\coker(\lambda') \to \coker(\varepsilon)$.
     \end{enumerate}
\end{lemma}

\begin{proof}
  (a)${}\Leftrightarrow{}$(b): 
  We have $\pi_M\lambda(B_R)=\varepsilon\pi_N(B_R)= \varepsilon(N_R)$.
  It follows that $\pi_M^{-1}(\varepsilon(N_R)) = \lambda(B_R)+\ker\pi_M$. Thus (a) is equivalent to $\ker\pi_M\subseteq \lambda(B_R)$. The inclusion $\lambda(\ker(\pi_N)) \subseteq \ker(\pi_M)$ always holds by the commmutativity of the diagram, so that $b$ is equivalent to $\ker(\pi_M)\subseteq\lambda(\ker(\pi_N))$. Thus (b)${}\Rightarrow{}$(a) is trivial. Conversely, if (a) holds, and $a\in\ker(\pi_M)$, then $a=\lambda(b)$ for some $b\in B_R$, so that $0=\pi_M(a)=\pi_M\lambda(b)=\varepsilon\pi_N(b)$. But $\varepsilon$ is mono, so $\pi_N(b)=0$, and $a=\lambda(b)\in\lambda(\ker(\pi_N))$.

  (b)${}\Leftrightarrow{}$(c)
    Apply the Snake Lemma to the diagram
    \[
      \xymatrix{ 0\ar[r] & \ker(\pi_N)\ar[r]\ar[d]\ar[d]^{\lambda|_{\ker}} & B_R\ar[d]^{\lambda} \ar[r]^{\pi_N} & N_R \ar[d]_{\varepsilon} \ar[r]& 0\\ 
      0\ar[r] & \ker(\pi_M)\ar[r] & A_R \ar[r]_{\pi_M} & M_R \ar[r]& 0,}
    \]
    obtaining a short exact sequence
    \[
    \xymatrix{ 0=\ker(\varepsilon)\ar[r] & \coker({\lambda|_{\ker}})\ar[r] & \coker({\lambda}) \ar[r] & \coker({\varepsilon})\ar[r] & 0.}
     \]
    Therefore $\lambda(\ker(\pi_N))=\ker(\pi_M)$ if and only if $\lambda|_{\ker}$ is surjective, if and only if $\coker({\lambda|_{\ker}})=0$, if and only if the epimorphism $\coker(\lambda)\to\coker(\varepsilon)$ is injective, if and only if it is an isomorphism.

  Now assume that there exist isomorphisms $\varphi_A\colon A_R' \to A_R$ and $\varphi_B\colon B_R' \to B_R$ and set $\pi_N' := \pi_N \varphi_B$, $\pi_M' := \pi_M \varphi_A$ and $\lambda' := \varphi_A^{-1} \lambda \varphi_B$. To conclude the proof, it suffices to show that $\lambda(\ker(\pi_N)) = \ker(\pi_M)$ if and only if $\lambda'(\ker(\pi_N')) = \ker(\pi_M')$. This is true, since $\ker(\pi_M') = \varphi_A^{-1}(\ker(\pi_M))$ and
    \[
    \lambda'(\ker(\pi_N')) = \lambda'(\varphi_B^{-1}(\ker(\pi_N))) = \varphi_A^{-1} \lambda \varphi_B (\varphi_B^{-1}(\ker(\pi_N))) = \varphi_A^{-1}(\lambda(\ker(\pi_N))). \qedhere
    \]
\end{proof}

\section{$\pi$-exactness}

Let $M_R$ be a cyclically presented right $R$-module and $\pi_M\colon R_R \to M_R$ a cyclic presentation. We introduce the notion of $\pi_M$-exactness to characterize those submodules of $M_R$ that lift, via $\pi_M$, to principal right ideals of $R$, generated by a left cancellative element of $R$. We give sufficient conditions on $R$ for this notion to be independent from the chosen presentation $\pi_M$.

\begin{deflemma}[$\pi$-exactness]
  \label{piexact}
  Let $N_R \le M_R$ be cyclic right $R$-modules. Let $F_R \cong R_R$, fix an epimorphism $\pi_M\colon F_R \to M_R$ and let $\varepsilon\colon N_R \hookrightarrow M_R$ denote the embedding.
  The following conditions are equivalent:
  \begin{enumerate}
    \tolerance=900 % otherwise we have lots of overfull hboxes here
    \enumequivi

  \item\label{piexact:a} $\pi_M^{-1}(N_R) \cong R_R$.

  \item\label{piexact:b}
      There exists a monomorphism $\lambda\colon R_R \to F_R$ and an epimorphism $\pi_N\colon R_R \to N_R$ such that $\lambda(\ker(\pi_N)) = \ker(\pi_M)$ and the following diagram commutes:
    \begin{equation}
      \begin{gathered}
      \xymatrix{ R_R \ar[r]^{\lambda} \ar[d]_{\pi_N} &
      F_R \ar[d]^{\pi_M} \\
      N_R \ar[r]_{\varepsilon} & M_R. }
      \label{eq:piexact-1}
      \end{gathered}
    \end{equation}

  \item \label{piexact:c} There exists a monomorphism $\lambda\colon R_R \to F_R$ and an epimorphism $\pi_N\colon R_R \to N_R$ such that diagram \eqref{eq:piexact-1} commutes and induces an isomorphism $\coker(\lambda) \to \coker(\varepsilon)$.
  \end{enumerate}

  If these equivalent conditions are satisfied, we call $N_R$ a \emph{$\pi_M$-exact submodule} of~$M_R$.
\end{deflemma}

\begin{proof}
  \ref*{piexact:a}${}\Rightarrow{}$\ref*{piexact:b}.
  By \ref*{piexact:a}, there exists an isomorphism $\lambda_0\colon R_R \to \pi_M^{-1}(N_R)$. Let $\lambda$ be the composite mapping $R_R \xrightarrow{\lambda_0} \pi_M^{-1}(N_R) \hookrightarrow F_R$ and $\varepsilon^{-1}\colon \varepsilon(N_R)\to N_R$ be the inverse of the corestriction of $\varepsilon$ to $\varepsilon(N_R)$. Noticing that $\pi_M\lambda(R_R)=\varepsilon(N_R)$, one gets an onto mapping  $\pi_N := \varepsilon^{-1} \pi_M \lambda\colon R_R \to N_R$.
  Then diagram \eqref{eq:piexact-1} clearly commutes and $\lambda(R_R) = \pi_M^{-1}(N_R)$. The statement now follows from Lemma \ref*{square}.

  \ref*{piexact:b}${}\Leftrightarrow{}$\ref*{piexact:c} and \ref*{piexact:b}${}\Rightarrow{}$\ref*{piexact:a}. By Lemma \ref*{square}.
\end{proof}

\begin{cor}
  Let $F_R \cong R_R$ and let $\pi_M\colon F_R \to M_R$ be an epimorphism. If $\varphi\colon F'_R \to F_R$ is an isomorphism and $N_R \le M_R$, then $N_R$ is a $\pi_M$-exact submodule of $M_R$ if and only if it is a $\pi_M \varphi$-exact submodule of $M_R$. 
\end{cor}

\begin{proof}
  Let $N_R$ be a $\pi_M$-exact submodule of $M_R$ and let $\lambda\colon R_R \to F_R$ be a monomorphism satisfying condition \ref{piexact:b} of Definition and Lemma \ref{piexact}. Apply Lemma \ref*{square} to $B_R = B_R' = R_R$, $A_R = F_R$, $A_R' = F_R'$, $\varphi_B = 1_R$  and $\varphi_A = \varphi$. Setting $\lambda' :=\varphi^{-1}\lambda$, it follows that $\lambda'(\ker(\pi_N)) = \ker(\pi_M \varphi)$ and hence $N_R$ is a $\pi_M \varphi$-exact submodule of $M_R$. The converse follows applying what we have just shown to $\varphi^{-1}$.
\end{proof}

\begin{cor}  \label{cor:quo-cp}
  Let $N_R \le M_R$ be cyclic $R$-modules, $\pi_M\colon R_R \to M_R$ be an epimorphism and $N_R \le M_R$ be a $\pi_M$-exact submodule.
  Then $M_R/N_R$ is cyclically presented with presentation induced by $\pi_M$.
\end{cor}

\begin{proof}
  Let $\lambda\colon R_R \to R_R$ be as in condition (c) of Definition and Lemma \ref{piexact}.
  Then $M_R / N_R \cong R_R / \lambda(R_R)$, from which the conclusion follows immediately.
\end{proof}

\begin{cor} \label{cor:transitive}
  Let $N_R \le M_R \le P_R$ be cyclic $R$-modules and let $\pi_P\colon F_R \to P_R$ be an epimorphism, where $F_R \cong R_R$.
  If $M_R \le P_R$ is $\pi_P$-exact and $N_R \le M_R$ is $\pi_P|_{\pi_P^{-1}(M_R)}$-exact, then $N_R \le P_R$ is $\pi_P$-exact.
\end{cor}

\begin{proof}
  Set $F'_R := \pi_P^{-1}(M_R)$. By condition \ref{piexact:a} of Definition and Lemma \ref{piexact}, $F'_R \cong R_R$.
  Therefore the notion of $\pi_P|_{F'_R}$-exactness of $N_R$ in $M_R$ is indeed defined.
  Since $\pi_P^{-1}(N_R) = (\pi_P|_{F'_R})^{-1}(N_R) \cong R_R$, the claim follows.
\end{proof}

Let $c \in R$ be left cancellative and denote by $\sL(cR, R)$ the set of all right ideals $aR$ with $a \in R$ left cancellative and $cR \subset aR \subset R$. It is partially ordered by set inclusion. Let $\pi\colon R \to R/cR$ be an epimorphism. Denote by $\sL_{\pi}(R/cR)$ the set of all $\pi$-exact submodules of $R/cR$. This set is also partially ordered by set inclusion.

\begin{lemma} \label{lemma:bij-posets}
  Let $c \in R$ be left cancellative and let $\pi\colon R_R \to R/cR$ be the canonical epimorphism. Then $\pi$ induces an isomorphism of partially ordered sets $\sL(cR, R) \cong \sL_{\pi}(R/cR)$.
\end{lemma}

\begin{proof}
  It suffices to show that $N_R \subset R/cR$ is $\pi$-exact if and only if there exists a left cancellative $a \in R$ with $\pi^{-1}(N_R) = aR$. But this is equivalent to $\pi^{-1}(N_R) \cong R_R$. The statement now follows from condition Definition and Lemma \ref{piexact:a} of  \ref{piexact}.
\end{proof}

The following example shows that, in general, the condition of $\pi$-exactness indeed depends on the particular choice of the epimorphism $\pi\colon R_R \to M_R$.
We refer the reader to any of \cite{Maclachlan-Reid '03}, \cite{Reiner '75} or \cite{Vigneras '80} for the necessary background on quaternion algebras.

\begin{example}{\rm 
  Let $A$ be a quaternion algebra over $\bQ$ and $R$ be a maximal $\bZ$-order in $A$ such that there exists an unramified prime ideal $\fP \subset R$ and maximal right ideals $I, J$ of $R$ with $I, J \supset \fP$, $I$ principal and $J$ non-principal.
  Then $\fp = \fP \cap \bZ$ is principal, say $\fp = p\bZ$ with $p \in \bP$, $\fP = pR$,  $R / \fP \cong M_2(\bF_p)$ and $\fP = \Ann(R/\fP)$.
  (E.g., take $A = \left( \f{-1,-11}{\bQ} \right)$, $R = {}_\bZ \langle 1, i, \f{1}{2}(i+j), \f{1}{2}(1+k) \rangle$, $p=3$, $I = {}_\bZ \langle \f{1}{2}(1+5k), \f{1}{2}(i + 5j), 3j, 3k \rangle$ and $J = {}_\bZ \langle \f{1}{2}(1 + 2j + 3k), \f{1}{2}(i + 3j + 4k), 3j, 3k \rangle$).

  The module $R/\fP$ has a composition series (as an $R/\fP$- and hence as an $R$-module)
  \[
     0 \subsetneq I/\fP \subsetneq R/\fP,
  \]
  and there exists an isomorphism $R/\fP \to R/\fP$ mapping $J/\fP$ to $I/\fP$, as is easily seen from $R/\fP \cong M_2(\bF_p)$. 
  Therefore there exist epimorphisms $\pi_M\colon R \to R/\fP$ and $\pi_M'\colon R \to R/\fP$ with $\pi_M^{-1}(I/\fP) = I$ and $\pi_M'^{-1}(I/\fP) = J$. This implies that $I/\fP$ is a $\pi_M$-exact submodule of $R/ \fP$ that is not $\pi_M'$-exact.}
\end{example}

However, under an additional assumption on $R_R$, which holds, for instance, whenever $R$ is a semilocal ring, the notion is independent of the choice of $\pi$. 

\begin{lemma}
  Suppose that $R_R \oplus K_R \cong R_R \oplus R_R$ implies $K_R \cong R_R$ for all right ideals $K_R$ of $R$.
  \begin{enumerate}
    \item If $M_R \cong R/aR$ with $a \in R$ left cancellative and $\pi_M\colon R_R \to M_R$ is an epimorphism, then there exists a left cancellative $a' \in R$ such that $\ker(\pi_M) = a'R$.
    \item If $M_R$ is a cyclic $R$-module, $\pi_M\colon R_R \to M_R$ and $\pi_M'\colon R_R \to M_R$ are epimorphisms and $N_R \le M_R$, then $N_R$ is a $\pi_M$-exact submodule of $M_R$ if and only if it is a $\pi_M'$-exact submodule of $M_R$.
  \end{enumerate}
\end{lemma}

\begin{proof}
(1) Let $\pi_{aR}\colon R_R \to R/aR, 1 \mapsto 1 + aR$ be the canonical epimorphism. Since $a$ is left cancellative, $aR \cong R_R$.
      Consider the exact sequences
      \[
         0 \to aR \hookrightarrow R_R \xrightarrow{\pi_{aR}} R/aR \to  0
      \]
      and
      \[
         0 \to \ker(\pi_M) \hookrightarrow R_R \xrightarrow{\pi_M} R/aR \to  0.
      \]
      By Schanuel's Lemma, $R_R \oplus aR \cong R_R \oplus \ker(\pi_M)$, and hence by assumption $aR \cong \ker(\pi_M)$. Thus there exists a left cancellative $a' \in R$ with $\ker(\pi_M) = a'R$.

(2) Let $\pi_{M/N}\colon M_R \to M_R/N_R$ be the canonical quotient module epimorphism.
      There are exact sequences
      \[
         0 \to \pi_M^{-1}(N_R) \to R_R \xrightarrow{\pi_{M/N} \pi_M} M_R/N_R \to  0
      \]
      and
      \[
         0 \to \pi_M'^{-1}(N_R) \to R_R \xrightarrow{\pi_{M/N} \pi_M'} M_R/N_R \to  0,
      \]
      and by Schanuel's Lemma therefore $R_R \oplus \pi_M^{-1}(N_R) \cong R_R \oplus \pi_M'^{-1}(N_R)$.
        If $N_R$ is a $\pi_M$-exact submodule of $M_R$, then $\pi_M^{-1}(N_R) \cong R_R$ and hence $\pi_M'^{-1}(N_R) \cong R_R$ by our assumption on $R$, showing that $N_R$ is a $\pi_M'$-exact submodule. The converse follows by symmetry.
\end{proof}

Suppose that $R$ has invariant basis number (for all $m,n \in \mathbb N_0$, $R_R^m \cong R_R^n$ implies $m = n$).
Then the condition of the previous lemma is satisfied if every stably free $R$-module of rank $1$ is free \cite[\S11.1.1]{McConnell-Robson '01}. This is true if $R$ is commutative \cite[\S11.1.16]{McConnell-Robson '01}.
The condition is also true if $R$ is semilocal \cite[Corollary 4.6]{Facchini '98} or $R$ is a $2$-fir (by \cite[Theorem 1.1(e)]{Cohn '85}).

Let $M_R$ be a right $R$-module with an epimorphism  $\pi_M\colon R_R \to M_R$ with $\ker(\pi_M) = aR$ and $a \in R$ left cancellative.
We say that a finite series
\[
   0 = M_0 \subset M_1 \subset M_2 \subset \ldots \subset M_n = M_R
\]
of submodules is {\em $\pi_M$-exact}, if every $M_i$ is an $\pi_M|_{\pi_M^{-1}(M_{i+1})}$-exact submodule of $M_{i+1}$. By Lemma \ref{lemma:bij-posets} the $\pi_M$-exact series of submodules of $R$ are in bijection with series of principal right ideals in $\sL(aR, R)$. By Lemma \ref{lemma:fact-chain} they are therefore in bijection with factorizations of $a$ into left cancellative elements, up to insertion of units.

\medskip

Recall that a ring $R$ is a $2$-fir if and only if it is a domain and the sum of any two principal right ideals with non-zero intersection is again a principal right ideal \cite[Theorem 1.5.1]{Cohn '85}. In the next theorem, we will consider, for a cyclically presented right $R$-module $M_R$ and a cyclic presentation $\pi_M\colon R_R \to M_R$ with non-zero kernel, the set of all submodules of cyclically presented $\pi_M$-exact submodules. We say it is {\em closed under finite sums} if for every two cyclically presented $\pi_M$-exact submodules $M_1$ and $M_2$ of $M_R$, the sum $M_1 + M_2$ also is cyclically presented and a $\pi_M$-exact submodule of $M_R$. 

\begin{theorem}\label{3.8}
  Let $R$ be a domain. The following are equivalent.
  \begin{enumerate}
    \item For every cyclically presented right $R$-module $M_R$ and every cyclic presentation $\pi_M\colon R_R \to M_R$ with non-zero kernel, the set of all cyclically presented $\pi_M$-exact submodules is closed under finite sums.    
    \item $R$ is a $2$-fir.
  \end{enumerate}
\end{theorem}

\begin{proof}
  (1) $\Rightarrow$ (2):
  Let $a,b,c \in R\setminus \{0\}$ be such that $cR \subset aR \cap bR$. We have to show that $aR + bR$ is right principal. Let $M_R = R/cR$, $\pi_M\colon R_R \to R/cR$ be the canonical epimorphism, $M_1 = aR / cR$ and $M_2 = bR / cR$. By Lemma \ref{lemma:bij-posets}, $M_1 = \pi_M(aR)$ and $M_2 = \pi_M(bR)$ are $\pi_M$-exact submodules of $M_R$. By assumption $M_1 + M_2$ is a $\pi_M$-exact submodule of $M_R$. Again by Lemma \ref{lemma:bij-posets}, $aR + bR = \pi_M^{-1}(M_1 + M_2)$ is a principal right ideal of $R$, generated by a left cancellative element.

  \smallskip
  (2) $\Rightarrow$ (1):
  We may assume $M_1, M_2 \ne  0$, as the statement is trivial otherwise.
  Let $\pi_M\colon R_R \to M_R$ be an epimorphism with non-zero kernel.
  Since $M_1$ and $M_2$ are $\pi_M$-exact submodules of $M_R$, there exist $a,b \in R\setminus \{0\}$ such that $\pi^{-1}(M_1) = aR$ and $\pi^{-1}(M_2) = bR$. Because $\ker(\pi) \ne  0$, we have $aR \cap bR \ne  0$. Since $R$ is a $2$-fir, there exists $c \in R\setminus \{0 \}$ such that $aR + bR = \pi_M^{-1}(M_1 + M_2) = cR$. Therefore $M_1 + M_2$ is cyclically presented and a $\pi_M$-exact submodule of $M_R$.
\end{proof}

Notice that if we assume that sums and intersections of exact submodules are again exact submodules, one may use the Artin-Schreier and Jordan-H\"older-Theorems to study factorizations of elements. As we have just seen, such an assumption leads to the $2$-firs investigated by Cohn in \cite{Cohn '85}.

\section{Projective covers of cyclically presented modules}

Let $R$ be a ring and $R/xR$ a cyclically presented right $R$-module, $x\in R$. The module $R/xR$ does not have a projective cover in general, but if it has one, it has one of the form $\pi|_{eR}\colon eR\to R/xR$, where $e\in R$ is an idempotent that depends on $x$ and $\pi|_{eR}$ is the restriction to $eR$ of the canonical projection $\pi\colon R_R\to R/xR$ \cite[Lemma~17.17]{AF2}. More precisely, given any projective cover $p\colon P_R\to R/xR$, there is an isomorphism $f\colon eR\to P_R$ such that $pf=\pi|_{eR}$. The kernel of the projective cover $\pi|_{eR}\colon eR\to R/xR$ is $eR\cap xR$ and is contained in $eJ(R)$ because the kernel of $\pi|_{eR}$ is a superfluous submodule of $eR$ and $eJ(R)$ is the largest superfluous submodule of $eR$.  Considering the exact sequences $0\to xR\to R_R\to R/xR\to 0$ and $0\to eR\cap xR\to eR\to R/xR\to 0$, one sees that $R_R\oplus (eR\cap xR)\cong eR\oplus xR$ (Schanuel's Lemma), so that $eR\cap xR$ can be generated with at most two elements.

Recall that every right $R$-module has a projective cover if and only if the ring $R$ is perfect, and that every finitely generated right $R$-module has a projective cover if and only every simple right $R$-module has a projective cover, if and only if the ring $R$ is semiperfect. Denoting by $J(R)$ the Jacobson radical of $R$, $R$ is semiperfect if and only if $R/J(R)$ is semisimple and idempotents can be lifted modulo $J(R)$ \cite[Theorem 27.6]{AF2}.
The next result gives a similar characterization for the rings $R$ over which every cyclically presented right module has a projective cover.

\begin{theorem}\label{VNR} The following conditions are equivalent for a ring $R$ with Jacobson radical $J(R)$:

{\rm (1)} Every cyclically presented right $R$-module has a projective cover.

{\rm (2)} The ring $R/J(R)$ is Von Neumann regular and idempotents can be lifted modulo $J(R)$.\end{theorem}

\begin{proof}
  Set $J := J(R)$.
  
  (1)${}\Rightarrow{}$(2) Assume that every cyclically presented right $R$-module has a projective cover. In order to show that $R/J$ is Von Neumann regular, it suffices to prove that every principal right ideal of $R/J$ is a direct summand of the right $R/J$-module $R/J$ \cite[Theorem~1.1]{Goodearl '91}. Let $x$ be an element of $R$. We will show that $(xR+J)/J$ is a direct summand of $R/J$ as a right $R/J$-module. By (1),  the cyclically presented right $R$-module $R/xR$ has a projective cover. As we have seen above, the projective cover is of the form $\pi|_{eR}\colon eR\to R/xR$ for some idempotent $e$ of $R$, where $\pi\colon R_R\to R/xR$ is the canonical projection.
  
  Applying the right exact functor $-\otimes_RR/J$ to the short exact sequence $0\to eR\cap xR\to eR\to R/xR\to 0$, we get an exact sequence $(eR\cap xR)\otimes_RR/J\to eR\otimes_RR/J\to R/xR\otimes_RR/J\to 0$, which can be rewritten as $(eR\cap xR)/(eR\cap xR)J\to eR/eJ\to R/(xR+J)\to 0$. It follows that there is a short exact sequence $0\to ((eR\cap xR)+eJ)/eJ\to eR/eJ\to R/(xR+J)\to 0$. Now the kernel $eR\cap xR$ of the projective cover $\pi|_{eR}$ is superfluous in 
$eR$ and $eJ$ is the largest superfluous submodule of $eR$, hence $((eR\cap xR)+eJ)/eJ=0$ and $eR/eJ\cong R/(xR+J)$. 

Now $(e+J)(R/J)=(eR+J)/J\cong eR/(eR\cap J)=eR/eJ$, so that $eR/eJ\cong R/(xR+J)$ is a projective right $R/J$-module. Thus the short exact sequence $0\to(x+J)(R/J)=(xR+J)/J\to R/J\to R/(xR+J)\to 0$ splits, and the principal right ideal of $R/J$ generated by $x+J$ is a direct summand of the right $R/J$-module $R/J$. 

We must now prove that idempotents of $R/J$ lift modulo $J$. By \cite[Proposition 27.4]{AF2}, this is equivalent to showing that every direct summand of the $R$-module $R/J$ has a projective cover. Let $M_R$ be a direct summand of $(R/J)_R$. Then it is also a direct summand of $(R/J)_{R/J}$ and hence is generated by an idempotent of $R/J$.
Let $g \in R$ be such that $g+J \in R/J$ is idempotent and $M_{R/J} = (g + J)(R/J)$. Then $R/J = (g + J)(R/J) \oplus (1 - g + J)(R/J)$ as $R/J$-modules, and hence also as $R$-modules. The canonical projection $\pi_g\colon R/J \to M_R$ has kernel $\ker(\pi_g) = (1-g + J)(R/J)$. Let $\pi\colon R_R \to R/J, r \mapsto r +J$ be the canonical epimorphism. Set $f: = \pi_g \pi$. Then $\ker(f) = (1-g)R + J$ and so $f$ factors through an epimorphism $\lbar{f}\colon R / (1-g)R \to M_R$ with $\ker(\lbar{f}) = (J + (1-g)R) / (1-g)R$. In particular, $\ker(\lbar{f})$ is the image of the superfluous submodule $J$ of $R_R$ via the canonical projection $R_R\to R/(1-g)R$. It follows that $\ker(\lbar{f})$ is superfluous in  $R/(1-g)R$, i.e.,  $\lbar f$ is a superfluous epimorphism.

By hypothesis, there is a projective cover $p\colon P_R \to R/(1-g)R$.
Since the composite mapping of two superfluous epimorphisms is a superfluous epimorphism (this follows easily from \cite[Corollary 5.15]{AF2}), $\lbar{f} p\colon P_R \to M_R$ is a superfluous epimorphism and hence a projective cover of $M$.

(2)${}\Rightarrow{}$(1) Assume that (2) holds. Let $R/xR$ be a cyclically presented right $R$-module, where $x\in R$. The principal right ideal $(x+J)(R/J)$ of the Von Neumann regular ring $R/J$ is generated by an idempotent and idempotents can be lifted modulo $J$. Hence there exists an idempotent element $e\in R$ such that $(x+J)(R/J)=(e+J)(R/J)$. Let $\pi|_{(1-e)R}$ be the restriction to $(1-e)R$ of the canonical epimorphism $\pi\colon R_R\to R/xR$. We claim that $\pi|_{(1-e)R}\colon (1-e)R\to R/xR$ is onto. To prove the claim, notice that $xR+J=eR+J$, so that $(1-e)R+xR+J=R$. As $J$ is superfluous in $R_R$, it follows that $(1-e)R+xR=R$ and so $\pi|_{(1-e)R}$ is onto. This proves our claim. Finally, $\ker(\pi|_{(1-e)R})=(1-e)R\cap xR\subseteq ((1-e)R+J)\cap(xR+J)=((1-e)R+J)\cap(eR+J)\subseteq J$, so that $\ker(\pi|_{(1-e)R})\subseteq J\cap (1-e)R=(1-e)J$ is superfluous in $(1-e)R$. Thus $\pi|_{(1-e)R}$ is the required projective cover of the cyclically presented $R$-module $R/xR$.
\end{proof}

\begin{cor} \label{cor:local-cov}
  If $R$ is a domain and every cyclically presented right $R$-module has a projective cover, then $R$ is local. 
\end{cor}

\begin{proof}
  By the previous Theorem, $R/J(R)$ is Von Neumann regular. Since idempotents lift modulo $J(R)$, the only idempotents of $R/J(R)$ are $0$ and $1$. Therefore $R/J(R)$ is a division ring and so $R$ is local\end{proof}

Notice that, conversely, if $R$ is a local ring and $M_R$ is  any non-zero cyclic module, then every epimorphism $\pi\colon R_R \to M_R$ is a projective cover.

\begin{lemma}\label{cuo} Let $R$ be an arbitrary ring, let $N_R\le M_R$ be cyclic right $R$-modules with a projective cover and let $\varepsilon\colon N_R \to M_R$ be the embedding.
  Then the following two conditions are equivalent:
  
  \begin{enumerate}
    \item There exist a projective cover $\pi_N\colon P_R\to N_R$ of $N_R$, a projective cover $\pi_M\colon Q_R\to M_R$ of $M_R$ and a commutative diagram of right $R$-module morphisms
    \begin{equation}
      \begin{gathered}
        \xymatrix{ P_R \ar[r]^{\lambda} \ar[d]_{\pi_N} &
        Q_R \ar[d]^{\pi_M} \\
        N_R \ar[r]_{\varepsilon} & M_R, }
        \label{1'}
      \end{gathered}
    \end{equation}
    such that the following equivalent conditions hold:
    \begin{enumerate}
      \item[{\rm (a)}] $\lambda(P_R) = \pi_M^{-1}(\varepsilon(N_R))$;
      \item[{\rm (b)}] $\lambda(\ker(\pi_N))=\ker(\pi_M)$;
      \item[{\rm (c)}] $\pi_M$ induces an isomorphism $\coker(\lambda)\to\coker(\varepsilon)$.
    \end{enumerate}
  \item For every pair of projective covers $\pi_N\colon P_R\to N_R$ of $N_R$ and $\pi_M\colon Q_R\to M_R$ of $M_R$ and every commutative diagram \eqref{1'} of right $R$-module morphisms, the following equivalent conditions hold:
    \begin{enumerate}
      \item[{\rm (a')}] $\lambda(P_R) = \pi_M^{-1}(\varepsilon(N_R))$;
      \item[{\rm (b')}] $\lambda(\ker(\pi_N))=\ker(\pi_M)$;
      \item[{\rm (c')}] $\pi_M$ induces an isomorphism $\coker(\lambda)\to\coker(\varepsilon)$.
    \end{enumerate}
  \end{enumerate}
\end{lemma}

\begin{proof}
  The equivalences (a)${}\Leftrightarrow{}$(b)${}\Leftrightarrow{}$(c) and (a')${}\Leftrightarrow{}$(b')${}\Leftrightarrow{}$(c') have been proved in Lemma \ref*{square}. 
  
  (b)${}\Rightarrow{}$(b'):
  Assume that $\pi_N\colon P_R\to N_R$, $\pi_M\colon Q_R\to M_R$ and $\lambda\colon P_R\to Q_R$ satisfy condition (b), that is, make diagram~(\ref{1'}) commute and $\lambda(\ker(\pi_N))=\ker(\pi_M)$. Let $\pi'_N\colon P'_R\to N_R$ and $\pi'_M\colon Q'_R\to M_R$ be projective covers and $\lambda'\colon P'_R\to Q'_R$ be a morphism that make the diagram corresponding to diagram~(\ref{1'}) commute, that is, such that $\pi'_M\lambda'=\varepsilon\pi'_N$. Projective covers are unique up to isomorphism and, by Lemma \ref*{square}, we may therefore assume $P'_R = P_R$, $Q'_R = Q_R$ and $\pi_M' = \pi_M$, $\pi_N' = \pi_N$. 

   Then $\pi_M(\lambda-\lambda')=\pi_M\lambda-\varepsilon\pi_N=\varepsilon\pi_N-\varepsilon\pi_N=0$, so that $(\lambda-\lambda')(P_R)\subseteq \ker\pi_M$.
   Let $\iota\colon\ker\pi_M\to Q_R$ denote the inclusion.
   Then there exists a morphism $\psi\colon P_R\to\ker\pi_M$ such that $\lambda-\lambda'=\iota\psi$. As images via module morphisms of superfluous submodules are superfluous submodules and $\ker\pi_N$ is a superfluous submodule of $P_R$, it follows that $\psi(\ker\pi_N)$ is a superfluous submodule of $\ker\pi_M$.
   Now $\ker\pi_M=\lambda(\ker\pi_N)=(\lambda'+\iota\psi)(\ker\pi_N)\subseteq\lambda'(\ker\pi_N)+\iota\psi(\ker\pi_N)=\lambda'(\ker\pi_N)+\psi(\ker\pi_N)\subseteq \ker\pi_M$. Thus $\ker\pi_M=\lambda'(\ker\pi_N)+\psi(\ker\pi_N)$. But $\psi(\ker\pi_N)$ is superfluous in $\ker\pi_M$, hence $\ker\pi_M=\lambda'(\ker\pi_N)$, which proves (b').

(b')${}\Rightarrow{}$(b): Let $\pi_N\colon P_R \to N_R$ and $\pi_M\colon Q_R \to M_R$ be projective covers of $N_R$, respectively $M_R$. Since $P_R$ is projective and $\pi_M\colon Q_R \to M$ is an epimorphism, there exists a $\lambda\colon P_R \to Q_R$ such that $\pi_M \lambda = \varepsilon \pi_N$. By (b'), then $\lambda(\ker(\pi_N)) = \ker(\pi_M)$.
\end{proof}

\begin{definition}
  If $N_R\le M_R$ are cyclic right $R$-modules and the equivalent conditions of Theorem \ref{cuo} are satisfied, we say that $N_R$ is an \emph{exact submodule} of $M_R$.
\end{definition}

\begin{cor}
  If $L_R\le M_R\le N_R$ are cyclic right $R$-modules, $M_R$ is exact in $N_R$ and $L_R$ is exact in $M_R$, then $L_R$ is exact in $N_R$.
\end{cor}

\begin{proof}
  Since $L_R$ is exact in $M_R$ and $M_R$ is exact in $N_R$, there exist projective covers $\pi_L\colon P_R \to L_R$, $\pi_M\colon Q_R \to M_R$, $\pi_M'\colon Q_R' \to M_R$ and $\pi_N\colon U_R \to N_R$ and homomorphisms $\lambda\colon P_R \to Q_R$ and $\mu\colon Q_R' \to U_R$ such that $\pi_M\lambda = \pi_L$, $\pi_N\mu = \pi_M'$, $\lambda(\ker(\pi_L)) = \ker(\pi_M)$ and $\mu(\ker(\pi_M')) = \ker(\pi_N)$.

  Since the projective cover of $M_R$ is unique up to isomorphism, we may assume by Lemma \ref*{square} that $Q_R = Q_R'$ and $\pi_M'=\pi_M$ (replacing $\lambda$ accordingly). Then $\pi_N \mu \lambda = \pi_M \lambda = \pi_L$ and $\ker(\pi_N) = \mu(\ker(\pi_M)) = \mu(\lambda(\ker(\pi_L)) = (\mu\lambda)(\ker(\pi_L))$. Therefore $N_R$ is an exact submodule of $M_R$.
\end{proof}

\begin{cor} \label{cor:cp}
If a cyclic module $N_R$ is an exact submodule of a cyclic module $M_R $ and $M_R$ has a projective cover isomorphic to $R_R$, then $M_R/N_R$ is cyclically presented.
\end{cor}
\begin{proof}
Since $N_R$ is an exact submodule of $M_R$, there exists a commutative diagram
\begin{equation*}
\xymatrix{ P_R \ar[r]^{\lambda} \ar[d]_{\pi_N} &
Q_R \ar[d]^{\pi_M} \\
N_R \ar[r]_{\varepsilon} & M_R}
\end{equation*}
where $\pi_N\colon P_R \to N_R$ and $\pi_M\colon Q_R \to M_R$ are projective covers of $N_R$ and $M_R$ and $\coker(\lambda) \cong \coker(\varepsilon)$. By assumption, there exists an idempotent $e \in R$ such that $P_R \cong eR$ and $Q_R \cong R_R$. By Lemma~\ref*{square}, we may therefore assume $P_R = eR$ and $Q_R = R_R$ (replacing $\pi_M$, $\pi_N$ and $\lambda$ accordingly). Therefore $M_R / N_R = \coker(\varepsilon) \cong \coker(\lambda) = R / eR$. Hence $M_R/N_R$ is cyclically presented.
\end{proof}

The following example shows that if $R$ is not a domain, then even if a non-unit $x \in R$ is not a zero-divisor, the projective cover of $R/xR$ need not be isomorphic to $R_R$.

\begin{example}{\rm 
Let $D$ be a discrete valuation ring and $\pi \in D$ a prime element. The unique maximal ideal of $D$ is $\pi D$. Let $R=M_2(D)$, $x=\begin{bmatrix} 1&0 \\ 0& \pi \end{bmatrix}$ and $e=\begin{bmatrix} 0&0 \\0&1\end{bmatrix}$.

We have
\[
  xR=\begin{bmatrix} D&D \\ \pi D& \pi D \end{bmatrix} \quad\text{and}\quad eR= \begin{bmatrix} 0&0 \\ D&D\end{bmatrix}.
\]
Let $p\colon R_R \to R/xR$ be the canonical projection. We will show that $p|_{eR}:eR \to R/xR$ is a projective cover of $R/xR$. We have $\ker p|_{eR}=xR \cap eR =\begin{bmatrix} 0&0 \\ \pi D& \pi D\end{bmatrix}$. Since $J(R)=M_2(J(D))=\begin{bmatrix}\pi D& \pi D \\ \pi D&\pi D\end{bmatrix}$, it follows that $\ker p|_{eR} = eJ(R)$.
Since $e$ is an idempotent of $R$, $eR$ is projective and $eJ(R)=J(eR)$. In particular, $\ker p|_{eR}$ is superfluous in $eR$. Therefore $eR$ is a projective cover of $R/xR$.

We now show that $eR \not \cong R$. Assume $eR$ is isomorphic to $R$. Then there exists an isomorphism $f\colon R_R \to eR$. Hence $f(1)=\begin{bmatrix} 0&0 \\ c&d \end{bmatrix} \ne 0$.

 Let $b=\begin{bmatrix}{-d}&0 \\ c&0\end{bmatrix}$. Then $b \ne 0$, because $f(1) \ne 0$. But $f(1)b=\begin{bmatrix} 0&0 \\ c&d \end{bmatrix} \begin{bmatrix}{-d}&0 \\ c&0 \end{bmatrix}=\begin{bmatrix} 0&0 \\0&0 \end{bmatrix}$ implies $f(b)=0$. It follows that $b=0$, which contradicts $b \ne 0$. Thus $eR$ is not isomorphic to $R$.}
\end{example}

The next example shows that the condition for the projective cover of $M_R$ to be isomorphic to $R_R$ is necessary in Corollary \ref{cor:cp}.

\begin{example}{\rm 
  Let $R=T_2(\mathbb{Z} /2\mathbb{Z})$ be the ring of all upper triangular $2 \times 2$ matrices with coefficients in $\Z/2\Z$. Since $J(R)$ consists of all strictly upper triangular matrices, $R/J(R) \cong (\mathbb Z / 2\mathbb Z)^2$ is semisimple and obviously idempotents lift modulo $J(R)$. Therefore every finitely generated $R$-module has a projective cover.
Set $$M_R :=\begin{bmatrix} 1&0 \\ 0&0 \end{bmatrix}R =\left\{
\begin{bmatrix} 
0&0 \\ 0&0 \end{bmatrix},
\begin{bmatrix}
0&1 \\ 0&0\end{bmatrix},
\begin{bmatrix}
1&0 \\ 0&0 \end{bmatrix},
\begin{bmatrix}
1&1 \\ 0&0\end{bmatrix}
\right\},$$

$$N_R:=\begin{bmatrix}
0&1 \\ 0&0\end{bmatrix}R
=\left\{ \begin{bmatrix}
0&0 \\ 0&0 \end{bmatrix},
\begin{bmatrix}
0&1 \\ 0&0\end{bmatrix}
\right\},$$

$$M_R/N_R=\left\{ \begin{bmatrix}
0&0 \\ 0&0 \end{bmatrix}+N_R, \begin{bmatrix}
1&0 \\ 0&0 \end{bmatrix}+N_R \right\}.$$
Consider 
\begin{eqnarray*}\phi\colon N_R &\longrightarrow& \begin{bmatrix} 0&0 \\ 0&1 \end{bmatrix}R\\
\begin{bmatrix} 0&c \\ 0&0 \end{bmatrix} &\longmapsto& \begin{bmatrix} 0&0 \\ 0&c \end{bmatrix}\end{eqnarray*}
It is obvious that $\phi$ is an isomorphism. Since $\begin{bmatrix} 0&0 \\ 0&1 \end{bmatrix}$ is an idempotent of $R$, $\begin{bmatrix} 0&0 \\ 0&1 \end{bmatrix}R$ is a projective $R$-module. Hence $N_R$ is a projective $R$-module. On the other hand, $M_R$ is also a projective $R$-module, because $\begin{bmatrix} 1&0 \\ 0&0 \end{bmatrix}$ is an idempotent of $R$. Hence $1_N\colon N_R \to N_R$ and $1_M\colon M_R \to M_R$ are projective covers. This implies that the diagram 
$$\xymatrix{ N_R \ar[r]^{\varepsilon} \ar[d]_{1_N} &
M_R \ar[d]^{1_M} \\
N_R \ar[r]_{\varepsilon} & M_R, }$$ where $\varepsilon(\ker 1_N) = \ker 1_M$, commutes. Therefore $N_R$ is an exact submodule of $M_R$.

Assume $M_R/N_R$ is a cyclically presented module. Then $M_R/N_R$ is isomorphic to $R/xR$, where $x \in R$. Since $\left|M_R/N_R\right|=2$, $\left|xR\right|=4$. We have

$$\begin{bmatrix} 0&0 \\ 0&0 \end{bmatrix}R= \left\{ \begin{bmatrix} 0&0 \\ 0&0 \end{bmatrix}\right\},$$
$$\begin{bmatrix} 0&1 \\ 0&0 \end{bmatrix}R=\left\{ \begin{bmatrix} 0&1 \\ 0&0 \end{bmatrix}
\begin{bmatrix} a&b \\ 0&c \end{bmatrix}=\begin{bmatrix} 0&c \\ 0&0 \end{bmatrix}\right \}=N_R,$$
$$\begin{bmatrix} 1&0 \\ 0&0\end{bmatrix}R=M_R,$$
$$\begin{bmatrix} 1&1 \\ 0&0 \end{bmatrix}R=\left\{ \begin{bmatrix} 1&1 \\ 0&0 \end{bmatrix} \begin{bmatrix} a&b \\ 0&c \end{bmatrix} =\begin{bmatrix} a&{b+c} \\ 0&0 \end{bmatrix} \right\}=M_R,$$
$$\begin{bmatrix} 0&0 \\ 0&1 \end{bmatrix}R= \left\{ \begin{bmatrix} 0&0 \\ 0&1 \end{bmatrix} \begin{bmatrix} a&b \\ 0&c \end{bmatrix} = \begin{bmatrix} 0&0 \\ 0&c \end{bmatrix} \right\},$$
$$\begin{bmatrix} 0&1 \\ 0&1 \end{bmatrix}R = \left\{ \begin{bmatrix} 0&1 \\ 0&1 \end{bmatrix} \begin{bmatrix} a&b \\ 0&c \end{bmatrix} = \begin{bmatrix} 0&c \\ 0&c \end{bmatrix} \right\},$$
$$\begin{bmatrix} 1&0 \\ 0&1 \end{bmatrix}R=R_R,$$
$$\begin{bmatrix} 1&1 \\ 0&1 \end{bmatrix}R=R_R.$$
Thus $xR=M_R$. Hence
$$R/xR=R/M_R=\left\{\begin{bmatrix} 0&0 \\ 0&0\end{bmatrix}+M_R, \begin{bmatrix} 1&0 \\ 0&1\end{bmatrix}+M_R\right\},$$

\begin{eqnarray*}
\ann(M_R/N_R)&=&\left\{ \begin{bmatrix} a&b \\ 0&c \end{bmatrix} \in R\; \Bigg\vert
\begin{bmatrix} 1&0 \\ 0&0 \end{bmatrix} \begin{bmatrix} a&b \\ 0&c\end{bmatrix} \in N_R \right\}\\
&=&\left\{\begin{bmatrix} a&b \\ 0&c\end{bmatrix} \in R\; \Bigg\vert \begin{bmatrix} a&b \\ 0&0 \end{bmatrix} \in N_R\right\}\\
&=& \left\{ \begin{bmatrix} 0&0 \\ 0&0 \end{bmatrix}, \begin{bmatrix} 0&1 \\ 0&0 \end{bmatrix}, \begin{bmatrix} 0&0 \\ 0&1 \end{bmatrix}, \begin{bmatrix} 0&1 \\ 0&1 \end{bmatrix}  \right\},
\end{eqnarray*}
$$\ann(R/xR)=\left\{ \begin{bmatrix} a&b \\ 0&c \end{bmatrix} \in R\; \Bigg\vert \begin{bmatrix} 1&0 \\ 0&1 \end{bmatrix} \begin{bmatrix} a&b \\ 0&c \end{bmatrix} \in xR=M_R \right\}=M_R.$$
Hence $\ann(M_R/N_R) \ne \ann(R/xR)$. On the other hand, we have $\ann(M_R/N_R)=\ann(R/xR)$ since $M_R/N_R$ is isomorphic to $R/xR$. This is a contradiction. Therefore $M_R/N_R$ is not a cyclically presented module.}
\end{example}

\begin{prop}
  Let $R$ be a local domain. Let $N_R, M_R \ne  0$ be cyclically presented right $R$-modules and let $\pi_M\colon R_R \to M_R$ be an epimorphism.
  Then $N_R \subset M_R$ is exact if and only if it is $\pi_M$-exact in the sense of Definition and Lemma \ref{piexact}.
\end{prop}

\begin{proof}
  Suppose first $N_R \subset M_R$ exact.
  Let $\pi_N\colon R_R \to N_R$ be any epimorphism. Then $\pi_M$ and $\pi_N$ are necessarily projective covers, because $\ker(\pi_M)$ and $\ker(\pi_N)$ are superfluous.
  Let $\varepsilon\colon N_R \to M_R$ denote the inclusion.
  By projectivity of $R_R$, there exists a $\lambda\colon R_R \to R_R$ such that $\pi_M \lambda = \varepsilon \pi_N$. By condition (a) in Lemma \ref{cuo}, $\lambda(R_R) = \pi_M^{-1}(N_R)$. Since $\pi_M^{-1}(N_R) \ne  0$, it follows that $\pi_M^{-1}(N_R) \cong R_R$ and hence condition (a) in Definition and Lemma \ref{piexact} is satisfied.

  Suppose now that $N_R \subset M_R$ is $\pi_M$-exact. Let $\pi_N\colon R_R \to N_R$ be an epimorphism and $\lambda\colon R_R \to R_R$ a monomorphism satisfying condition (b) of Definition and Lemma \ref{piexact}. Then $\pi_N$ is a projective cover of $N_R$, and condition (b) of Lemma \ref{cuo} is satisfied, implying that $N_R \subset M_R$ is exact.
\end{proof}

The previous proposition, together with the results from the previous section, shows that in the special case of $R$ a local domain and $x \in R$  a non-unit, series of exact submodules of $R/xR$ may be used to study factorizations of $x \in R$ up to insertion of units.

\section{Cokernels of endomorphisms}

Let $M_R$ be a right module over a ring $R$ and let $E:=\End(M_R)$ be its endomorphism ring. Let $s$ be a fixed element of $E$. In this section, we investigate the relation between projective covers $eE \to E/sE$ for an idempotent $e$, induced by the canonical epimorphism $E_E \to E/sE$, and properties of the module $e(M_R)$.
This is of particular interest if we assume that $E/J(E)$ is Von Neumann regular and idempotents can be lifted modulo $J(E)$, as in this case for every non-zero $s \in E$ the module $E/sE$ has a projective cover. For instance, every continuous module $M_R$ has this property \cite[Proposition~3.5 and Corollary~3.9]{SM}, in particular every quasi-injective module has this property, and every module of Goldie dimension one and dual Goldie dimension one has this property \cite[Proposition~2.5]{SM}.

Let $s\colon M_R\to M_R$ be an endomorphism of $M_R$. We can consider the direct summands $M_1$ of $M_R$ such that there exists a direct sum decomposition $M_R=M_1\oplus M_2$ of $M_R$ for some complement $M_2$ of $M_1$ with the property that $\pi_2s\colon M_R\to M_2$ is a split epimorphism. Here $\pi_2\colon M_R\to M_2$ is the canonical projection with kernel $M_1$. Let $\Cal F$ be the set of all such direct summands, that is,
\[
  \begin{split}
    \Cal F:= \{\,M_1\mid\; & M_1\le M_R,\ \text{ there exists }M_2\le M_R\ \text{such that }M_R=M_1\oplus M_2 \\
      & \text{and } \pi_2s\colon M_R\to M_2\ \text{ a split epimorphism}\,\}.
  \end{split}
\]
The set $\Cal F$ can be partially ordered by set inclusion.

It is well known that there is a one-to-one correspondence between the set of all pairs $(M_1,M_2)$ of $R$-submodules of $M_R$ such that $M_R=M_1\oplus M_2$ and the set of all idempotents $e\in E$. If $e\in E$ is an idempotent, the corresponding pair is the pair $(M_1:=e(M_R),M_2:=(1-e)(M_R))$.
If $s \in \End(M_R)$, we always denote by $\varphi\colon E_E \to E/sE$ the canonical epimorphism $\varphi(f) = f + sE$.

\begin{lemma} \label{lemma:surj-equi}
  Let $M_R = M_1 \oplus M_2$, let $\pi_2\colon M_R \to M_2$ be the projection with kernel $M_1$, and let $e \in \End(M_R)$ be the endomorphism corresponding to the pair $(M_1,M_2)$. If $s\colon M_R \to M_R$ is an endomorphism, then $\pi_2 s$ is a split epimorphism if and only if $\varphi|_{eE}\colon eE \to E/sE$ is surjective.
\end{lemma}

\begin{proof}
  We have to show that $\pi_2s\colon M_R\to M_2$ is a split epimorphism if and only if $eE+sE=E$.
  In order to prove the claim, assume that $\pi_2s\colon M_R\to M_2$ is a split epimorphism, so that there is an $R$-module morphism $f\colon M_2\to M_R$ with $\pi_2sf=1_{M_2}$. Let $\varepsilon_2\colon M_2\to M_R$ be the embedding. Then the right ideal $eE+sE$ of $E$ contains the endomorphism
  \[
    \begin{split}
    e(1_M-sf\pi_2)+s(f\pi_2)&=e+(1_M-e)sf\pi_2=e+\varepsilon_2\pi_2sf\pi_2 \\
                            &=e+\varepsilon_21_{M_2}\pi_2=e+(1_M-e)=1_M,
    \end{split}
  \]
  so that $eE+sE=E$. Conversely, let $e\in E$ be an idempotent with $eE+sE=E$, so that there exist $g,h\in E$ with $1=eg+sh$. Then $(1-e)=(1-e)sh$, so that $(1-e)=(1-e)sh(1-e)$, that is, $\varepsilon_2\pi_2=\varepsilon_2\pi_2sh\varepsilon_2\pi_2$. Since $\varepsilon_2$ is injective and $\pi_2$ is surjective, they can be canceled, so that $1_{M_2}=\pi_2sh\varepsilon_2$. Hence $\pi_2s$ is a split epimorphism, which proves our claim.
\end{proof}

\begin{prop} \label{prop:min-splitepi}
  Let $M_R$ be a right module, and let $E:=\End(M_R)$ be its endomorphism ring.
  Let $s\in E$ and suppose that $E/sE$ has a projective cover. Then
\[
  \begin{split}
    \Cal F:= \{\,M_1\mid\; & M_1\le M_R,\ \text{ there exists }M_2\le M_R\ \text{such that }M_R=M_1\oplus M_2 \\
      & \text{and } \pi_2s\colon M_R\to M_2\ \text{ a split epimorphism}\,\}
  \end{split}
\]
  has minimal elements, and all minimal elements of $\Cal F$ are isomorphic $R$-submodules of $M_R$.
\end{prop}

\begin{proof}
From the previous lemma, it follows that there is  a one-to-one correspondence between the set $\Cal F'$ of all pairs $(M_1,M_2)$ of $R$-submodules of $M_R$ such that $M_R=M_1\oplus M_2$ and $\pi_2s\colon M_R\to M_2$ is a split epimorphism and the set of all idempotents $e\in E$ for which the canonical mapping $eE\to  E_E/sE$, $x\in eE\mapsto x+sE$, is surjective. In order to prove that $\Cal F$ has minimal elements, it suffices to show that if the canonical mapping $eE\to  E_E/sE$ is a projective cover, then $e(M_R)$ is a minimal element of $\Cal F$. Let $e\in E$ be such that $eE\to  E_E/sE$ is a projective cover, and let $M'_1\in\Cal F$ be such that $M'_1\subseteq e(M_R)$.  Let $e'\in E$ be an idempotent such that $M'_1=e'(M_R)$ and $\pi'_2s\colon M_R\to (1-e')(M_R)$ is a split epimorphism. Then $M'_1=e'(M_R)\subseteq e(M_R)$, so that $ee'=e'$. Thus $e'E=ee'E\subseteq eE$. If $\varphi|_{eE}\colon eE\to E/sE$ is the projective cover, $\varphi|_{e'E}\colon e'E\to E/sE$ denotes the canonical epimorphism and $\varepsilon\colon e'E\to eE$ is the embedding, it follows that $\varphi|_{eE}\varepsilon=\varphi|_{e'E}$. Now $\varphi|_{eE}$ is a superfluous epimorphism and $\varphi|_{eE}\varepsilon=\varphi|_{e'E}$ is onto, so that $\varepsilon$ is onto, that is, $e'E=eE$. Thus $e=e'f$ for some $f\in E$, so that $e(M_R)\subseteq e'(M_R)=M'_1$ and $M'_1=e(M_R)$. It follows that $e(M_R)$ is a minimal element of $\Cal F$.

Now let $M''_1$ be any other minimal element of $\Cal F$, and let $e''$ be an idempotent element of $E$ with $\pi''_2s\colon M_R\to (1-e'')(M_R)$ a split epimorphism. Then the canonical projection $e''E\to E/sE$ is an epimorphism. As the canonical projection $\varphi|_{eE}\colon eE\to E/sE$ is the projective cover, there is a direct sum decomposition $e''E=P'_E\oplus P''_E$ with the canonical projection $P'_E\to E/sE$ a projective cover. Thus $P'_E=p'E$ for some idempotent $p'$ of $E$ with $p'E+sE=E$, so that $p'(M_R)\in \Cal F$. Now $e''E\supseteq P'_E=p'E$ implies that $p'=e''g$ for some $g\in E$, so that $p'(M_R)\subseteq e''(M_R)=M''_1$. By the minimality of $M''_1$ in $\Cal F$, it follows that $p'(M_R)=e''(M_R)$, so that $M''_1=e''(M_R)=p'(M_R)\cong p'E\otimes_EM_R=P'\otimes_EM_R\cong eE\otimes_EM_R\cong e(M_R)$. Thus every minimal element of $\Cal F$ is isomorphic to $e(M_R)$.\end{proof}

We conclude the paper by considering quasi-projective modules.
Let $M_R$ and $N_R$ be right $R$-modules. Recall that $M_R$ is \emph{$N_R$-projective} if for every epimorphism $f\colon N_R \to L_R$ and every homomorphism $g\colon M_R \to L_R$ there exists a homomorphism $h\colon M_R \to N_R$ such that $g=fh$.
Equivalently, for every epimorphism $f\colon N_R \to L_R$, the induced homomorphism $f_*\colon \Hom_R(M_R,N_R) \to \Hom_R(M_R,L_R)$ is surjective. If $M_R$ is $N_R$-projective and $K_R \le N_R$, then $M_R$ is also $K_R$-projective \cite[Proposition 16.12(1)]{AF2}. A right $R$-module $M_R$ is \emph{quasi-projective} if it is $M_R$-projective. Trivially, projective modules and semisimple modules are quasi-projective.

Let $M_R$ be quasi-projective, $E := \End_R(M_R)$ and suppose $s \in E$. In the following, we relate projective covers of the $R$-module $M_R/s(M_R)$ and the cyclically presented $E$-module $E/sE$.

\begin{lemma} \label{lemma:endo}
Let $M_R$ be a quasi-projective right R-module, $E$ the endomorphism ring of $M_R$ and let $s \in E$. Let $\pi$ be the canonical epimorphism of $M_R$ onto $M_R/s(M_R)$ and $\varphi$ the canonical epimorphism of $E_E$ onto $E/sE$.
\begin{enumerate}
  \item For every $g \in E$, $\pi|_{g(M_R)}$ is surjective if and only if $\varphi|_{gE}$ is surjective.
\item For every $g\in E$, $gE$ is a direct summand of $E_E$ if and only if $g(M_R)$ is a direct summand of $M_R$.
\item Let $e,e'$ be idempotents in $E$. Then $e(M_R) \cong e'(M_R)$ if and only if $eE \cong e'E$.
\item Let $e \in E$ be idempotent. Then $\ker(\pi|_{e(M_R)})$ is superfluous if and only if $\ker(\varphi|_{eE})$ is superfluous.
\end{enumerate}
\end{lemma}
\begin{proof}
  (1) (${}\Leftarrow{}$) Since $\varphi|_{gE}$ is surjective, there exists $h$ in $E$ such that $gh+sE=1_M+sE$. Hence there exists $h'$ in $E$ such that $gh=1_M+sh'$. For all $m \in M_R$ we have $\pi(m) = \pi(1_M(m)) = \pi(g(h(m))$, whence $\pi|_{g(M_R)}$ is surjective.
    
     (${}\Rightarrow{}$) 
     Since $M_R$ is quasi-projective and $\pi g\colon M_R \to M_R$ is an epimorphism, there exists $h\colon M_R \to M_R$ such that $\pi gh=\pi$.
Therefore $(gh-1_M)(M_R) \subset s(M_R)$.
Since $s\colon M_R \to s(M_R)$ is an epimorphism, quasi-projectivity of $M_R$ implies that there exists $h' \in E$ such that $gh-1_M=sh'$.
This implies that $\varphi(gh)=1_M+sE$. Therefore $\varphi|_{gE}$ is surjective.

\smallskip
(2) (${}\Rightarrow{}$) If $gE$ is a direct summand of $E$, there exists an idempotent $e$ in $E$ such that $gE=eE$. Hence there exist $h,h'$ in $E$ such that $g=eh$ and $e=gh'$. This implies that $g(M_R)=e(M_R)$. On the other hand, $e(M_R)$ is a direct summand of $M_R$ since $e$ is an idempotent of $E$. Therefore $g(M_R)$ is a direct summand of $M_R$.

  (${}\Leftarrow{}$) If $g(M_R)$ is a direct summand of $E$, there exists an idempotent $e$ in $E$ such that $g(M_R)=e(M_R)$. Hence $eg=g$. Therefore $gE \subset eE$.
  Since $g\colon M_R \to e(M_R)$ is an epimorphism and $M_R$ is quasi-projective, there exists $h\colon M_R \to M_R$ such that $e=gh$.
This implies that $eE \subset gE$. Hence $eE=gE$.

\smallskip
(3) (${}\Leftarrow{}$) 
  Since $eE \cong e'E$, there exists an isomorphism $\Gamma\colon eE \to e'E$. Consider the two following homomorphisms $f\colon e(M_R) \to e'(M_R)$ defined via $f(m)=e'x(m)$ where $e'x=\Gamma(e)$ and $g\colon e'(M_R) \to e(M_R)$ defined via $g(m)=ey(m)$ where $ey=\Gamma ^{-1}(e')$. It suffices to show that $fg=1_{e'(M_R)}$ and $gf=1_{e(M_R)}$. For $m \in e'(M_R)$,  $fg(m)= f(ey(m))=e'xey(m)=e'xy(m)=\Gamma(e)y(m)=\Gamma(ey)(m)=\Gamma(\Gamma^{-1}(e'))(m)=e'(m)=m$, it follows that $fg=1_{e'(M_R)}$. By an argument analogous to the previous one, we get $gf=1_{e(M_R)}$.

  (${}\Rightarrow{}$) Since $e(M_R) \cong e'(M_R)$, there exists an isomorphism $h\colon e(M_R) \to e'(M_R)$. Consider the two following homomorphisms $\theta\colon eE \to e'E$ defined via $\theta(ex)=e'hex$, and $\theta'\colon e'E \to eE$ defined via $\theta'(e'x)=eh^{-1}e'x$. It suffices to show that $\theta \theta'=1_{e'E}$ and $\theta' \theta =1_{eE}$. Since $\theta \theta'(e'x)(m)=\theta(eh^{-1}e'x)(m)=e'heh^{-1}e'x(m)=e'he(h^{-1}(e'x(m)))=e'h(h^{-1}(e'x(m)))=e'e'(x(m))=e'(x(m))$, it follows that $\theta \theta' (e'x)= e'x$. Hence $\theta \theta' =1_{e'E}$. By an argument analogous to the previous one, we get $\theta' \theta =1_{eE}$.

  \smallskip
  (4) (${}\Rightarrow{}$) Let $K_E$ be a submodule of $eE$ such that $K_E+\ker(\varphi|_{eE})=eE$. It suffices to show that $K_E=eE$. There exists $h \in \ker(\varphi|_{eE})=eE \cap sE$ and $k \in K_E$ such that $e=k+h$. Hence $e(M_R)=k(M_R)+h(M_R)$. This implies that $e(M_R)=k(M_R)+ \big(e(M_R) \cap s(M_R)\big)$. Since $e(M_R) \cap s(M_R)$ is superfluous in $e(M_R)$, then $e(M_R)=k(M_R)$.
  Since $k\colon M_R \to e(M_R)$ is an epimorphism and $M_R$ is quasi-projective, there exists $h'$ in $E$ such that $e=kh'$.
This implies that $e \in K_E$. Therefore $K_E=eE$.

(${}\Leftarrow{}$) Let $N_R$ be a submodule of $M_R$ such that $N_R+\ker(\pi|_{e(M_R)})=M_R$. Hence $\pi|_{N_R}$ is surjective. It suffices to show that $N_R=M_R$.
Since $M_R$ is quasi-projective and $N_R$ is a submodule of $M_R$, it follows that $M_R$ is also $N_R$-projective. Therefore the induced homomorphism $(\pi|_{N_R})_*\colon \Hom(M_R,N_R)\to \Hom(M_R,M_R/s(M_R))$ is surjective and hence there exists $g\colon M_R \to N_R$ such that $\pi g=\pi e$.
Again by quasi-projectivity of $M_R$, there exists $h\colon M_R \to M_R$ such that $g-e=sh$.
Since $g(M_R) \subset N_R \subset e(M_R)$, for every $x \in M_R$ there exists $y \in M_R$ such that $g(x)=e(y)$. We have $eg(x)=e(e(y))=e(y)=g(x)$. Thus $eg=g$. Since $g-e=eg-e=sh$, $eg-e \in eE$ and $sh \in sE$, it follows that $g-e \in eE \cap sE$. From $e=g-(g-e)$, we have $eE=gE+(g-e)E$. Hence $eE=gE+(eE \cap sE)$. Since $eE \cap sE= \ker \varphi|_{eE}$ is superfluous, $eE=gE$. Therefore $e(M_R)=g(M_R) \subset N_R$. Thus $N_R=e(M_R)$. \qedhere
\end{proof}
\begin{cor}
Let $M_R$ be a projective right $R$-module and $E$ the endomorphism ring of $M_R$. Let $s \in E$ and let $\pi$ be the canonical epimorphism from $M_R$ to $M_R/s(M_R)$ and $\varphi$ the canonical epimorphism from $E$ to $E/sE$.
Then $\pi|_{e(M_R)}$ is a projective cover of $M_R/s(M_R)$ if and only if $\varphi|_{eE}$ is a projective cover of $E/sE$.
\end{cor}
\begin{proof}
  Since $M_R$ is projective, so is $e(M_R)$. Hence $\pi|_{e(M_R)}$ is a projective cover if and only if $\ker(\pi|_{e(M_R)})$ is superfluous. Therefore the corollary follows from the previous lemma.
\end{proof}

\begin{prop}\label{vhlpp}
  Let $M_R$ be a quasi-projective right $R$-module, let $s \in E=\End(M_R)$ and let $\pi\colon M_R \to M_R/s(M_R)$ be the canonical epimorphism. Suppose that $E/sE$ has a projective cover.

  Consider $\cE := \{\, N_R \le M_R \mid \text{$\pi|_{N_R}$ is surjective} \,\}$ and $\cE_{\oplus} := \{\, N_R \in \cE \mid \text{$N_R$ is a direct summand of $M_R$}\}$, both partially ordered by set inclusion.
  Then $\cE_{\oplus}$ has minimal elements, any two minimal elements of $\cE_{\oplus}$ are isomorphic as right $R$-modules and any minimal element of $\cE_{\oplus}$ is minimal in $\cE$.
\end{prop}

\begin{proof}
  Let $N_R \le M_R$ be a direct summand of $M_R$, let $e \in E$ be an idempotent with $e(M_R)=N_R$ and let $\pi_2\colon M_R \to \ker(e)$ be the canonical projection corresponding to the direct sum decomposition $M_R = N_R \oplus \ker(e)$.
  Lemma \ref{lemma:endo}(1) implies that $\pi|_{N_R}\colon N_R \to M_R/s(M_R)$ is surjective if and only if $\varphi|_{eE}\colon eE \to E/sE$ is surjective.
  By Lemma \ref{lemma:surj-equi} this is the case if and only if $\pi_2 s$ is a split epimorphism.
  This shows that $\cE_{\oplus} = \cF$, where the latter is defined as in Proposition \ref{prop:min-splitepi}.
  The claims about $\cE_{\oplus}$ therefore follow from the proposition.

  It remains to show that the minimal elements of $\cE_{\oplus}$ are minimal in $\cE$. Let $N_R \in \cE_{\oplus}$ be minimal, and let $e\colon M_R \to N_R$ be an idempotent with $e(M_R) = N_R$. From the proof of Proposition \ref{prop:min-splitepi}, we see that $eE \to E/sE$ is a projective cover. Therefore Lemma \ref{lemma:endo}(4) implies that $\ker(\pi|_{N_R})$ is superfluous.
  Therefore, if $L_R \le N_R$ and $\pi|_{L_R}$ is surjective, we have $L_R + \ker(\pi|_{N_R}) = N_R$ and hence $L_R = N_R$, showing that $N_R$ is minimal in $\cE$.
\end{proof}

\subsection*{Acknowledgements} We are grateful to the anonymous referee for pointing out and correcting a mistake in an earlier version of the proof of Lemma \ref{lemma:endo}(4).

\end{document}